\documentclass[11pt,reqno]{amsart}

\usepackage{enumerate}
\usepackage{xspace}
\usepackage{amsmath,amsthm,amsfonts,amssymb,latexsym,mathrsfs,color}
\usepackage{hyperref}
\usepackage{graphicx}

\newcommand{\blst}{\begin{trivlist}}
\newcommand{\elst}{\end{trivlist}}

\newcommand{\op}[1]{\operatorname{#1}}

\newtheorem{thm}{Theorem}[section]
\newtheorem{prop}[thm]{Proposition}
\newtheorem{cor}[thm]{Corollary}
\newtheorem{lem}[thm]{Lemma}

\newtheorem{conj}[thm]{Conjecture}

\newtheorem{defn}[thm]{Definition}

\newtheorem{problem}[thm]{Problem}

\newcommand{\fc}{{\rm fcyc}}

\newcommand{\Des}{{\rm Des\,}}

\newcommand{\aexc}{{\rm aexc\,}}

\newcommand{\des}{{\rm des\,}}

\newcommand{\cyc}{{\rm cyc\,}}
\newcommand{\fbk}{{\rm fbk\,}}
\newcommand{\bk}{{\rm bk\,}}
\newcommand{\bkone}{{\rm bkone\,}}
\newcommand{\exc}{{\rm exc\,}}

\newcommand{\RLMAX}{\op{RLMAX}}
\newcommand{\rlmax}{\op{rlmax}}
\newcommand{\msn}{\mathfrak{S}_n}

\newcommand{\Q}{\mathcal{Q}}

\newcommand{\fix}{{\rm fix\,}}

\newcommand{\asc}{{\rm asc\,}}
\newcommand{\Asc}{{\rm Asc\,}}

\newcommand{\Eulerian}[2]{\genfrac{<}{>}{0pt}{}{#1}{#2}}
\newcommand{\stirling}[2]{\genfrac{[}{]}{0pt}{}{#1}{#2}}
\newcommand{\Stirling}[2]{\genfrac{\{}{\}}{0pt}{}{#1}{#2}}
\newcommand{\arxiv}[1]{\href{http://arxiv.org/abs/#1}{\texttt{arXiv:#1}}}
\newcommand{\red}{{\rm red\,}}
\title{Recurrence relations for binomial-Eulerian polynomials}
\author[J.~Ma]{Jun Ma}
\address{Department of mathematics, shanghai jiao tong university, shanghai, china}
\email{majun904@sjtu.edu.cn(J.~Ma)}
\author[S.-M.~Ma]{Shi-Mei Ma}
\address{School of Mathematics and Statistics,
        Northeastern University at Qinhuangdao,
         Hebei 066004, P.R. China}
\email{shimeimapapers@163.com (S.-M. Ma)}
\author[Y.-N. Yeh]{Yeong-Nan Yeh}
\address{Institute of Mathematics,
        Academia Sinica, Taipei, Taiwan}
\email{mayeh@math.sinica.edu.tw (Y.-N. Yeh)}
\subjclass[2010]{Primary 05A15; Secondary 05A19}

\begin{document}

\maketitle
\begin{abstract}
Binomial-Eulerian polynomials were introduced by Postnikov, Reiner and Williams.
In this paper, properties of the binomial-Eulerian polynomials, including recurrence relations and generating functions are studied. We present three constructive proofs of the recurrence relations for binomial-Eulerian polynomials. Moreover, we give a combinatorial interpretation of the Betti number of the complement of the $k$-equal real hyperplane arrangement.
\bigskip\\
{\sl Keywords:} Binomial-Eulerian polynomials; Eulerian polynomials; Recurrence relations
\end{abstract}
\date{\today}
\section{Introduction}
Let $\msn$ denote the symmetric group of all permutations of $[n]$, where $[n]=\{1,2,\ldots,n\}$.
For each $\pi\in\msn$, an index $i$ is called {\it a descent} (resp. {\it an ascent}) of $\pi$ if $\pi(i)>\pi(i+1)$ (resp. $\pi(i)<\pi(i+1)$), where $i\in [n-1]$.
Define
\begin{align*}
\Des(\pi)&=\{\pi(i)\mid \pi(i)>\pi(i+1),i\in [n-1]\},~\des(\pi)=|\Des(\pi)|,\\
\Asc(\pi)&=\{\pi(i)\mid \pi(i)<\pi(i+1),i\in [n-1]\},~\asc(\pi)=|\Asc(\pi)|,
\end{align*}
where $|S|$ denote the cardinality of the set $S$.
The classical {\it Eulerian polynomials} $A_n(x)$ are defined by
\begin{equation}\label{Eulerian}
A_n(x)=\sum_{\pi\in\msn}x^{\des(\pi)}=\sum_{\pi\in\msn}x^{\asc(\pi)}.
\end{equation}
Let $A_n(x)=\sum_{k=0}^{n-1}\Eulerian{n}{k}x^k$, where $\Eulerian{n}{k}$ are called the Eulerian numbers.
The numbers $\Eulerian{n}{k}$ satisfy the recurrence relation
$$\Eulerian{n}{k}=(k+1)\Eulerian{n-1}{k}+(n-k)\Eulerian{n}{k},$$
with the initial conditions $\Eulerian{1}{0}=1$ and $\Eulerian{1}{k}=0$ for $k\geq 1$ (see~\cite[A008292]{Sloane}).
In~\cite{Chung10}, Chung, Graham and Knuth noted that if we set $\Eulerian{0}{0}=0$, then
the following symmetrical identity holds:
\begin{equation}\label{Eulerian01}
\sum_{k\geq 0}\binom{a+b}{k}\Eulerian{k}{a-1}=\sum_{k\geq 0}\binom{a+b}{k}\Eulerian{k}{b-1},
\end{equation}
where $a,b$ are positive integers. Subsequently, the $q$-generalizations of the identity~\eqref{Eulerian01} have been
pursued by several authors. See, e.g.,~\cite{Chung12,Han11,Lin13,Shareshian}.

Let $G=K_{1,n}$ be the $n$-star graph with the central node $n+1$
connected to the nodes $1,\cdots,n$. The associated polytope $P_{{\mathcal B}(K_{1,n})}$ is called
the {\it stellohedron}.
Following~\cite[Section~10.4]{Postnikov}, the $h$-polynomial of the $n$-dimensional stellohedron is given by
\begin{equation}\label{binomial-Eulerian}
h_{{\mathcal B}(K_{1,n})}(x)=1+x\sum_{k=1}^n\binom{n}{k}A_k(x),
\end{equation}
which is named as the {\it binomial-Eulerian polynomial} (see~\cite{Shareshian}).
As usual, let $$\widetilde{A}_n(x)=h_{{\mathcal B}(K_{1,n})}(x).$$
The $\gamma$-positivity of $\widetilde{A}_n(x)$ follows from a general result of Postnikov, Reiner
and Williams~\cite[Theorem~11.6]{Postnikov}. As an application of the $\gamma$-positivity, we see that
$\widetilde{A}_n(x)$ is symmetric.
Very recently, Shareshian and Wachs~\cite{Shareshian} further studied $\gamma$-positivity of the binomial-Eulerian and $q$-binomial-Eulerian polynomials,
and noticed that the identity~\eqref{Eulerian01} is equivalent to the symmetry of $\widetilde{A}_n(x)$. The
reader is referred to~\cite{Athanasiadis} for a survey of the theory of $\gamma$-positivity.

\begin{defn}
Let $\Q_n$ be the set of permutations of $[n]$ with the restriction that the entry $n$ appears as the first descent. For convenience, let the identity permutation $12\cdots n$ be an element of $\Q_n$ and we say that the entry $n$ appears as the first descent of $12\cdots n$ (In fact, the identity permutation has no descent).
\end{defn}
For example, $\Q_1=\{1\},\Q_2=\{12,21\}$ and $\Q_3=\{123,132,231,312,321\}$.
Postnikov, Reiner and Williams~\cite[Section~10.4]{Postnikov} discovered that
$$\widetilde{A}_n(x)=\sum_{\pi\in\Q_{n+1}}x^{\des(\pi)}.$$
The first few of $\widetilde{A}_n(x)$ are given as follows:
\begin{align*}
\widetilde{A}_0(x)=1,
\widetilde{A}_1(x)=1+x,
\widetilde{A}_2(x)=1+3x+x^2,
\widetilde{A}_3(x)=1+7x+7x^2+x^3.
\end{align*}

It is clear that the ascent and descent statistics are equidistributed on $\msn$,
since reversing an element of $\msn$ turns ascents into descents and vice versa.
It is less obvious that ascent and descent statistics are equidistributed on $\Q_n$, since
reversing an element of $\Q_n$ may leads to an element of $\msn\backslash{\Q_n}$.

This paper is motivated by the following problem.
\begin{problem}\label{problem}
Is there a bijective proof of the symmetry of $\widetilde{A}_n(x)$ by using the descent and ascent statistics on $\Q_n$?
\end{problem}

This paper is organized as follows.
In Section~\ref{section02}, we present three constructive proofs of the recurrence relations for $\widetilde{A}_n(x)$.
In Theorem~\ref{equidistributed-bijection}, as a combination of the first two constructive proofs, we give a solution to Problem~\ref{problem}.
In Section~\ref{section03}, we study the generating function of a kind of multivariable binomial-Eulerian polynomials.
As an application, in Theorem~\ref{Betticombinatorial}, we give a combinatorial interpretation of the Betti number of the complement of the $k$-equal real hyperplane arrangement.
\section{Recurrence relations}\label{section02}
\subsection{The descent statistic on $\Q_n$}
\hspace*{\parindent}

It is well known that the Eulerian polynomials $A_n(x)$ satisfy the recurrence relation
$$A_{n+1}(x)=(1+nx)A_n(x)+x(1-x)A_n'(x),$$
with the initial values $A_0(x)=A_1(x)=1$ (see~\cite{Bre00} for instance), and they
can be defined by the exponential generating function
$$A(x,z)=\sum_{n\geq 0}A_n(x)\frac{z^n}{n!}=\frac{x-1}{x-e^{z(x-1)}}.$$
It is easy to verify that
\begin{equation}\label{Euler-recurrence01}
(1-xz)\frac{\partial}{\partial z}A(x,z)=A(x,z)+x(1-x)\frac{\partial}{\partial x}A(x,z).
\end{equation}
Set $\widetilde{A}_0(x)=1$.
We define $\widetilde{A}(x,z)=\sum_{n\geq 0}\widetilde{A}_n(x)\frac{z^n}{n!}$.
It follows from~\eqref{binomial-Eulerian} that
\begin{equation}\label{Axz-exz}
\widetilde{A}(x,z)=e^{xz}A(x,z).
\end{equation}
Combining~\eqref{Euler-recurrence01} and~\eqref{Axz-exz}, we obtain
\begin{equation}\label{Euler-recurrence02}
(1-xz)\frac{\partial}{\partial z}\widetilde{A}(x,z)=(1+x-xz)\widetilde{A}(x,z)+x(1-x)\frac{\partial}{\partial x}\widetilde{A}(x,z).
\end{equation}
Let $\widetilde{A}_n(x)=\sum_{k=0}^n\widetilde{A}(n,k)x^k$. Equating the coefficients of $x^kz^n/{n!}$ in both sides of~\eqref{Euler-recurrence02} leads to the following result.
\begin{thm}\label{thm01}
For $n\geq1$, we have
\begin{equation}\label{Enk-recurrence}
\widetilde{A}(n+1,k)=(k+1)\widetilde{A}(n,k)+(n-k+2)\widetilde{A}(n,k-1)-n\widetilde{A}(n-1,k-1),
\end{equation}
with the initial conditions $\widetilde{A}(0,0)=1$ and $\widetilde{A}(0,k)=0$ for $k\neq 0$.
\end{thm}

In the following, we present a constructive proof of the recurrence relation~\eqref{Enk-recurrence}.
Let $\alpha_i(\pi)$ be the permutation in $\mathfrak{S}_{n-1}$ obtained from $\pi$ by the following two steps:
\begin{itemize}
\item Step 1. Delete the entry $i$ from $\pi$;
\item Step 2. Every entry in $\pi$, which is larger than $i$, is decreased by $1$.
\end{itemize}
Let $\beta_{i,j}(\pi)$ be the permutation in $\mathfrak{S}_{n+1}$ obtained from $\pi$ by the following two steps:
\begin{itemize}
\item Step 1. Every entry in $\pi$, which is larger than or equal to $i$, is increased by $1$;
\item Step 2. Insert the entry $i$ between $j$-st and ($j$+$1$)-st elements of $\pi$.
\end{itemize}
In the sequel, we define
$$\Des^*(\pi)=\{0\}\cup \Des(\pi),$$
$$QD_{n,k}=\{\pi\in \Q_{n}\mid \des(\pi)=k\}.$$

Denote by $FD_{n+1,k}$ the set of pairs $[\pi,i]$ such that $\pi\in QD_{n+1,k}$ and $i\in \{0,1,2,\ldots,k\}$.
Hence $$|{FD}_{n+1,k}|=(k+1)\widetilde{A}(n,k).$$
We use ${RD}_{n+2,k}$ to denote the set of permutations $\pi$ of $[n+2]$ which satisfy the following three conditions:
\begin{enumerate}[(1)]
\item the entry $n+2$ appears as the first descent of $\pi$ from left to right;
\item $\pi$ has $k$ descents;
\item Either $a=1$ or $\pi(a-1)>\pi(a+1)$, where $a=\pi^{-1}(1)$.
\end{enumerate}

\begin{lem}\label{(k+1)E(n,k)-bijection}
There is a bijection $\phi=\phi_{n,k}$ from ${RD}_{n+2,k}$ to ${FD}_{n+1,k}$.
\end{lem}
\begin{proof}
For any $\pi\in {RD}_{n+2,k}$, let $a=\pi^{-1}(1)$ and $\sigma=\alpha_1(\pi)$. Clearly, $\sigma\in {QD}_{n+1,k}$. Suppose that $$\Des^*(\sigma)=\{j_0,j_1,\ldots,j_{k}\}$$ with $j_0<j_1<\ldots<j_{k-1}$ and $j_k=0$.
Note that  $a-1\in \Des^*(\sigma)$. Suppose that $j_i=a-1$ for some $i\in\{0,1,\ldots,k\}$. Define a map $\phi:{RD}_{n+2,k}\mapsto {FD}_{n+1,k}$ by letting $\phi(\pi)=[\alpha_1(\pi),i]$.

Conversely, for any $[\sigma,i]\in {FD}_{n+1,k}$, suppose that $\Des^*(\sigma)=\{j_0,j_1,\ldots,j_{k}\}$ with $j_0<j_1<\ldots<j_{k-1}$, $j_k=0$ and $a=j_i$. Let us consider the permutation $\pi=\beta_{1,a}(\sigma)$. Then $\pi(1)=1$ if $a=0$; otherwise, $\pi(a+1)=1$ and $\pi(a)>\pi(a+2)$ since $\sigma(a)>\sigma(a+1)$. So, $\pi\in {RD}_{n+2,k}$. Thus, for any $[\sigma,i]\in {FD}_{n+1,k}$, the inverse $\phi^{-1}$ of the map $\phi$ is given by $\phi^{-1}(\sigma,i)=\beta_{1,a}(\sigma)$.
\end{proof}

Let ${HD}_{n+1,k-1}$ be the set of pairs $[\pi,i]$ such that $\pi\in{QD}_{n+1,k-1}$ and $i\in \{1,2,\ldots,n-k+2\}$.
Then $$|{HD}_{n+1,k-1}|=(n-k+2)\widetilde{A}({n,k-1}).$$
Denote by ${RHD}_{n+1,k-1}$ the set of pairs $[\pi,i]$ such that $[\pi,i]\in{HD}_{n+1,k-1}$ and $i> \pi^{-1}(n+1)-1$.
We use $\overline{{RD}}_{n+2,k}$ to denote the set of permutations $\pi$ of $[n+2]$ which satisfy the following three conditions:
\begin{enumerate}[(1)]
\item the entry $n+2$ appears as the first descent of $\pi$ from left to right;
\item $\pi$ has $k$ descents;
\item Either $a=n+2$ or $\pi(a-1)<\pi(a+1)$, where $a=\pi^{-1}(1)$.
\end{enumerate}
\begin{lem}\label{bar-D(n+2,k-1)-bijection}
There is a bijection $\theta=\theta_{n,k}$ from $\overline{{RD}}_{n+2,k}$ to ${RHD}_{n+1,k-1}$.
\end{lem}
\begin{proof}
For any $\pi\in\overline{{RD}}_{n+2,k}$, let $a=\pi^{-1}(1)$ and $\sigma=\alpha_1(\pi)$. Clearly, $\sigma\in{{QD}}_{n+1,k-1}$ and $\asc(\sigma)=n-k+1$. Suppose that $$\Asc^{*}(\sigma)=\{j_1,j_2,\ldots,j_{n-k+2}\}$$ with $j_1<j_2<\ldots<j_{n-k+2}=n+1$. Note that $a-1\in \Asc^{*}(\sigma)$. Suppose that $j_i=a-1$ for some $i\in\{1,2,\ldots,n-k+2\}$. Then $i> \sigma^{-1}(n+1)-1$. Define a map $\theta:\overline{{RD}}_{n+2,k}\mapsto {RHD}_{n+1,k-1}$ by letting $\theta(\pi)=[\alpha_1(\pi),i]$.

Conversely, for any $[\sigma,i]\in {RHD}_{n+1,k-1}$, we have $\asc(\sigma)=n-k+1$. Suppose that $$\Asc^{*}(\sigma)=\{j_1,j_2,\ldots,j_{n-k+2}\}$$ with $j_1<j_2<\ldots<j_{n-k+2}=n+1$ and $a=j_i$. Let us consider the permutation $\pi=\beta_{1,a}(\sigma)$.
  Thus, $\pi^{-1}(1)=n+2$ if $a=n+1$; otherwise, $\pi(a+1)=1$ and $\pi(a)<\pi(a+2)$ since $\sigma(a)<\sigma(a+1)$. Hence $\pi\in\overline{{RD}}_{n+2,k}$ since $i>\sigma^{-1}(n+1)-1$. Therefore, the inverse $\theta^{-1}$ of the map $\theta$ is $\theta^{-1}(\sigma,i)=\beta_{1,a}(\sigma)$ for any $[\sigma,i]\in {RHD}_{n+1,k-1}$.
\end{proof}

Let $\overline{{HD}}_{n,k-1}$ be the set of pairs $[\pi,a]$ such that $\pi\in{QD}_{n,k-1}$ and $a\in \{1,2,\ldots,n\}$.
Then $$|\overline{{HD}}_{n,k-1}|=n\widetilde{A}({n-1,k-1}).$$
Let $\overline{{RHD}}_{n+1,k-1}={HD}_{n+1,k-1}\setminus {RHD}_{n+1,k-1}$. In fact, $\overline{{RHD}}_{n+1,k-1}$ is the set of pairs $[\pi,i]$ such that $[\pi,i]\in{HD}_{n+1,k-1}$ and $i\in\{1,2,\ldots,\pi^{-1}(n+1)-1\}$.

\begin{lem}\label{nE(n-1,k-1)-bijection}
There is a bijection $\psi=\psi_{n,k}$ from $\overline{{HD}}_{n,k-1}$ to $\overline{{RHD}}_{n+1,k-1}$.
\end{lem}
\begin{proof}
For any $[\sigma,a]\in \overline{{HD}}_{n,k-1}$, suppose $p=\sigma^{-1}(n)$, then $0=\sigma(0)<\sigma(1)<\sigma(2)<\ldots<\sigma(p)=n$ since the entry $n$ appears as the first descent of $\sigma$ from left to right. There exists a unique index $i\in\{0,1,\ldots,p-1\}$ such that $\sigma(i)<a\leq \sigma(i+1)$ since $a\in\{1,2,\ldots,n\}$.
Then $\beta_{a,i}(\sigma)\in QD_{n+1,k-1}$ and $[\beta_{a,i}(\sigma),i+1]\in \overline{{RHD}}_{n+1,k-1}$. Define a map $\psi:\overline{{HD}}_{n,k-1}\mapsto \overline{{RHD}}_{n+1,k-1}$ by letting $\psi(\sigma,a)=[\beta_{a,i}(\sigma),i+1]$.

Conversely, for any $[\sigma,i]\in\overline{{RH}}_{n+1,k-1}$, suppose $a=\sigma(i)$, then $a\in \{1,2,\ldots,n\}$ since the entry $n+1$ appears as the first descent of $\sigma$ from left to right and $i<\sigma^{-1}(n+1)$.
Moreover, $\alpha_{a}(\sigma)\in {QD}_{n,k-1}$ and $\alpha_{a}(\sigma)(i-1)<a\leq \alpha_{a}(\sigma)(i)$. The inverse $\psi^{-1}$ of the map $\psi$ is $$\psi^{-1}(\sigma,i)=[\alpha_a(\sigma),a].$$

\end{proof}

\noindent{\it The proof of the recurrence relation~\eqref{Enk-recurrence}:}

Note that $${QD}_{n+2,k}={RD}_{n+2,k}\cup \overline{{RD}}_{n+2,k}.$$
So $$\widetilde{A}(n+1,k)=|{QD}_{n+2,k}|=|{RD}_{n+2,k}|+ |\overline{{RD}}_{n+2,k}|.$$
Lemma~\ref{(k+1)E(n,k)-bijection} implies that $|{RD}_{n+2,k}|=|{FD}_{n+1,k}|=(k+1)\widetilde{A}(n,k)$.
Lemmas~\ref{bar-D(n+2,k-1)-bijection} and~\ref{nE(n-1,k-1)-bijection} tell us that
\begin{eqnarray*}|\overline{{RD}}_{n+2,k}|&=&|{RHD}_{n+1,k-1}|\\
&=&|{HD}_{n+1,k-1}|-|\overline{{RHD}}_{n+1,k-1}|\\
&=&|{HD}_{n+1,k-1}|-|\overline{{HD}}_{n,k-1}|\\
&=&(n-k+2)\widetilde{A}_{n,k-1}-n\widetilde{A}_{n-1,k-1}.\end{eqnarray*}
Hence, $\widetilde{A}(n+1,k)=(k+1)\widetilde{A}(n,k)+(n-k+2)\widetilde{A}(n,k-1)-n\widetilde{A}(n-1,k-1).$\hspace{2cm}$\square$
\begin{cor}\label{cor-zeros}
The polynomials $\widetilde{A}_n(x)$ satisfy the recurrence relation
\begin{equation*}\label{Anx-recurrence}
\widetilde{A}_{n+1}(x)=\left(1+(n+1)x\right)\widetilde{A}_n(x)+x(1-x)\widetilde{A}_n'(x)-nx\widetilde{A}_{n-1}(x),
\end{equation*}
with the initial value $\widetilde{A}_0(x)=1$.
\end{cor}

Based on empirical evidence, we propose the following conjecture.
\begin{conj}
For any $n\geq 1$, the polynomial $\widetilde{A}_n(x)$ has only real zeros.
\end{conj}
\subsection{The ascent statistic on $\Q_{n}$}
\hspace*{\parindent}

\begin{thm}\label{equidistributed}
We have $\widetilde{A}(n,k)=|\{\pi\in\Q_{n+1}: \asc(\pi)=k\}|$.
\end{thm}

Along the same lines of the proof of Theorem~\ref{thm01}, we shall present a constructive proof of Theorem~\ref{equidistributed}.

For any $n\geq 1$ and $\pi\in\mathfrak{S}_{n}$,
we define $$\Asc^*(\pi)=\{n\}\cup \Asc(\pi),$$
$${QA}_{n,k}=\{\pi\in \mathcal{Q}_{n}\mid \asc(\pi)=k\}.$$

Suppose  that the number of permutations in $\mathcal{Q}_{n+1}$ with $k$ ascents is $\widetilde{B}(n,k)$.
Let ${HA}_{n+1,k-1}$ be the set of pairs $[\pi,i]$ such that $\pi\in{QA}_{n+1,k-1}$ and $i\in \{1,2,\ldots,n-k+2\}$.
Then $$|{HA}_{n+1,k-1}|=(n-k+2)\widetilde{B}({n,k-1}).$$

We use ${{RA}}_{n+2,k}$ to denote the set of permutations $\pi$ of $[n+2]$ which satisfy the following three conditions:
\begin{enumerate}[(1)]
\item the entry $n+2$ appears as the first descent of $\pi$ from left to right;
\item $\pi$ has $k$ ascents;
\item Either $a=1$ or $\pi(a-1)>\pi(a+1)$, where $a=\pi^{-1}(1)$.
\end{enumerate}
\begin{lem}\label{asc-bar-D(n+2,k-1)-bijection}
There is a bijection $\hat{\theta}=\hat{\theta}_{n,k}$ from ${{RA}}_{n+2,k}$ to ${HA}_{n+1,k-1}$.
\end{lem}
\begin{proof}
For any $\pi\in{{RA}}_{n+2,k}$, let $a=\pi^{-1}(1)$ and $\sigma=\alpha_1(\pi)$. Clearly, $\sigma\in{{QA}}_{n+1,k-1}$ and $\des(\sigma)=n-k+1$. Suppose that $$\Des^{*}(\sigma)=\{j_1,j_2,\ldots,j_{n-k+2}\}$$ with $j_1<j_2<\ldots<j_{n-k+1}$ and $j_{n-k+2}=0$. Note that $a-1\in \Des^{*}(\sigma)$. Suppose that $j_i=a-1$ for some $i$. Define a map $\hat{\theta}:{{RA}}_{n+2,k}\mapsto {HA}_{n+1,k-1}$ by letting $\hat{\theta}(\pi)=[\alpha_1(\pi),i]$.

Conversely, for any $[\sigma,i]\in {HA}_{n+1,k-1}$, we have $\des(\sigma)=n-k+1$. Suppose that $$\Des^{*}(\sigma)=\{j_1,j_2,\ldots,j_{n-k+2}\}$$ with $j_1<j_2<\ldots<j_{n-k+1}$,  $j_{n-k+2}=0$ and $a=j_i$. Let us consider the permutation $\pi=\beta_{1,a}(\sigma)$. Then $\pi^{-1}(1)=1$ if $a=0$; otherwise, $\pi(a+1)=1$ and $\pi(a)>\pi(a+2)$ since $\sigma(a)>\sigma(a+1)$. Hence $\pi\in{{RA}}_{n+2,k}$. Therefore, the inverse $\hat{\theta}^{-1}$ of the map $\hat{\theta}$ is $$\hat{\theta}^{-1}(\sigma,i)=\beta_{1,a}(\sigma)$$ for any $[\sigma,i]\in {HA}_{n+1,k-1}$.
\end{proof}

Denote by ${FA}_{n+1,k}$ the set of pairs $[\pi,i]$ such that $\pi\in{QA}_{n+1,k}$ and $i\in \{1,\ldots,k\}\cup\{n+1\}$. Hence $$|{FA}_{n+1,k}|=(k+1)\widetilde{B}(n,k).$$
 Let ${{RFA}}_{n+1,k}$ be the set of pairs $[\pi,i]$ in ${FA}_{n+1,k}$ such that $i>\pi^{-1}(n+1)-1$.
Use $\overline{{RA}}_{n+2,k}$ to denote the set of permutations $\pi$ of $[n+2]$ which satisfy the following three conditions:
\begin{enumerate}[(1)]
\item the entry $n+2$ appears as the first descent of $\pi$ from left to right;
\item $\pi$ has $k$ ascents;
\item Either $a=n+2$ or $\pi(a-1)<\pi(a+1)$, where $a=\pi^{-1}(1)$.
\end{enumerate}

\begin{lem}\label{asc-(k+1)E(n,k)-bijection}
There is a bijection $\hat{\phi}=\hat{\phi}_{n,k}$ from ${\overline{{RA}}}_{n+2,k}$ to ${RFA}_{n+1,k}$.
\end{lem}
\begin{proof}
For any $\pi\in\overline{{RA}}_{n+2,k}$, let $a=\pi^{-1}(1)$ and $\sigma=\alpha_1(\pi)$.  Clearly, $\sigma\in{QA}_{n+1,k}$. Suppose that $$\Asc^{*}(\sigma)=\{j_0,j_1,\ldots,j_{k}\}$$ with $j_0<j_1<\ldots<j_{k-1}<j_{k}=n+1$. Note that $a-1\in \Asc^{*}(\sigma)$. Moreover, suppose that $j_i=a-1$ form some $i$. Then $i>\sigma^{-1}(n+1)-1$ since $a>\pi^{-1}(n+2)$. Define a map $\hat{\phi}:\overline{{RA}}_{n+2,k}\mapsto {RFA}_{n+1,k}$ by letting $\hat{\phi}(\pi)=[\alpha_1(\pi),i]$.

Conversely, for any $[\sigma,i]\in {RFA}_{n+1,k}$, suppose that $$\Asc^{*}(\sigma)=\{j_0,j_1,\ldots,j_{k}\}$$ with $j_0<j_1<\ldots<j_{k-1}< j_{k}=n+1$ and $a=j_i$. Let us consider the permutation $\pi=\beta_{1,a}(\sigma)$. Then $\pi(n+2)=1$ if $a=n+1$; otherwise, $\pi(a+1)=1$ and $\pi(a)<\pi(a+2)$ since $\sigma(a)<\sigma(a+1)$. So, $\pi\in\overline{{RA}}_{n+2,k}$ since $i>\sigma^{-1}(n+1)-1$. Thus, for any $[\sigma,i]\in {RFA}_{n+1,k}$, the inverse $\hat{\phi}^{-1}$ of the map $\hat{\phi}$ is given by $\hat{\phi}^{-1}(\sigma,i)=\beta_{1,a}(\sigma)$.
\end{proof}

Let $\overline{{HA}}_{n,k-1}$ be the set of pairs $[\pi,a]$ such that $\pi\in{QA}_{n,k-1}$ and $a\in \{1,2,\ldots,n\}$.
Then $$|\overline{{HA}}_{n,k-1}|=n\widetilde{B}({n-1,k-1}).$$
Let $\overline{{RFA}}_{n+1,k}={FA}_{n+1,k}\setminus {RFA}_{n+1,k}$. Note that $\overline{{RFA}}_{n+1,k}$ is the set of pairs $[\pi,i]$ such that $[\pi,i]\in{FA}_{n+1,k}$ and $i\in\{1,2,\ldots, \pi^{-1}(n+1)-1\}$.

\begin{lem}\label{asc-nE(n-1,k-1)-bijection}
There is a bijection $\hat{\psi}=\hat{\psi}_{n,k}$ from $\overline{{HA}}_{n,k-1}$ to $\overline{{RFA}}_{n+1,k}$.
\end{lem}
\begin{proof}
For any $[\sigma,a]\in \overline{{HA}}_{n,k-1}$, suppose $p=\sigma^{-1}(n)$, then $0=\sigma(0)<\sigma(1)<\sigma(2)<\ldots<\sigma(p)=n$ since the entry $n$ appears as the first descent of $\sigma$ from left to right. There exists a unique index $i\in\{0,1,\ldots,p-1\}$ such that $\sigma(i)<a\leq \sigma(i+1)$ since $a\in[n]$.
Then $\beta_{a,i}(\sigma)\in {QA}_{n+1,k}$ and $[\beta_{a,i}(\sigma),i+1]\in \overline{{RFA}}_{n+1,k}$. Define a map $\hat{\psi}:\overline{{HA}}_{n,k-1}\mapsto \overline{{RFA}}_{n+1,k}$ by letting $$\hat{\psi}(\sigma,a)=[\beta_{a,i}(\sigma),i+1].$$

Conversely, for any $[\sigma,i]\in\overline{{RFA}}_{n+1,k}$, suppose $a=\sigma(i)$, then $a\in [n]$ since the entry $n+1$ appears as the first descent of $\sigma$ from left to right and $i\leq \sigma^{-1}(n+1)-1$.
Moreover, $\alpha_{a}(\sigma)\in {QA}_{n,k-1}$ and $\alpha_{a}(\sigma)(i-1)<a\leq \alpha_{a}(\sigma)(i)$. The inverse $\hat{\psi}^{-1}$ of the map $\hat{\psi}$ is $\hat{\psi}^{-1}(\sigma,i)=[\alpha_a(\sigma),a]$.

\end{proof}

\noindent{\it The proof of the theorem \ref{equidistributed}:}

Note that ${QA}_{n+2,k}={RA}_{n+2,k}\cup \overline{{RA}}_{n+2,k}$.
So $\widetilde{B}(n+1,k)=|{QA}_{n+2,k}|=|{RA}_{n+2,k}|+ |\overline{{RA}}_{n+2,k}|$.
Lemma~\ref{asc-bar-D(n+2,k-1)-bijection} implies that $|{RA}_{n+2,k}|=|{HA}_{n+1,k}|=(n-k+2)\widetilde{B}(n,k-1)$.
Lemmas~\ref{asc-(k+1)E(n,k)-bijection} and~\ref{asc-nE(n-1,k-1)-bijection} tell us that \begin{eqnarray*}|\overline{{RA}}_{n+2,k}|&=&|{RFA}_{n+1,k}|\\
&=&|{FA}_{n+1,k}|-|\overline{{RFA}}_{n+1,k}|\\
&=&|{FA}_{n+1,k}|-|\overline{{HA}}_{n,k-1}|\\
&=&(k+1)\widetilde{B}_{n,k}-n\widetilde{B}_{n-1,k-1}.\end{eqnarray*}
Thus $\widetilde{B}(n+1,k)=(k+1)\widetilde{B}(n,k)+(n-k+2)\widetilde{B}(n,k-1)-n\widetilde{B}(n-1,k-1)$
 and so $\widetilde{B}(n,k)$ has the same recursion as $\widetilde{A}(n,k)$.
It is easy to check that $\widetilde{B}(0,0)=\widetilde{A}(0,0)=1$, $\widetilde{B}(1,0)=\widetilde{A}(1,0)=1$ and $\widetilde{B}(1,1)=\widetilde{A}(1,1)=1$. Hence $\widetilde{B}(n,k)=\widetilde{A}(n,k).$ \hspace{4cm}$\square$

\begin{thm}\label{equidistributed-bijection}
There is a bijection $\Omega_{n}$ from $\mathcal{Q}_{n}$ to itself such that $\des(\pi)=\asc(\Omega_n(\pi))$.
\end{thm}
\begin{proof} we can give a recursive definition of the bijection $\Omega_{n}$. For $n=1$, we have $\mathcal{Q}_1=\{1\}$. Let $\Omega_1(1)=1$. For $n=2$, we have $\mathcal{Q}_2=\{12,21\}$. Let $\Omega_2(12)=21$ and $\Omega_2(21)=12$.

For any $m=1,2,\ldots,n+1$, suppose that $\Omega_m$ is a bijection from $\mathcal{Q}_{m}$ to itself such that $\des(\pi)=\asc(\Omega_m(\pi))$ for any $\pi\in\mathcal{Q}_{m}$. Furthermore, for any pair $[\pi,i]$ with $\pi\in \mathcal{Q}_{m}$ and a nonnegative integer $i$, we let $$\hat{\Omega}_m(\pi,i)=[\Omega_m(\pi),i]\text{ and }\hat{\Omega}_m^{-1}(\pi,i)=[\Omega_m^{-1}(\pi),i].$$

For any $\pi\in\mathcal{Q}_{n+2}$, suppose that $\pi\in\mathcal{QD}_{n+2,k}$ for some $k$.  Note that $${QD}_{n+2,k}={RD}_{n+2,k}\cup \overline{{RD}}_{n+2,k}.$$
Combing the bijections in Lemmas~\ref{(k+1)E(n,k)-bijection},~\ref{bar-D(n+2,k-1)-bijection},~\ref{nE(n-1,k-1)-bijection},
~\ref{asc-bar-D(n+2,k-1)-bijection},~\ref{asc-(k+1)E(n,k)-bijection} and~\ref{asc-nE(n-1,k-1)-bijection} and the induction hypothesis, we give the bijection $\Omega_{n+2}$ from $\mathcal{Q}_{n+2}$ to itself as follows:
\begin{itemize}
  \item [($c_1$)] If $\pi\in{RD}_{n+2,k}$ and $\hat{\Omega}_{n+1}\circ\phi(\pi)\in{RFA}_{n+1,k}$,
  then let $$\Omega_{n+2}(\pi)=\hat{\phi}^{-1}\circ\hat{\Omega}_{n+1}\circ\phi(\pi);$$
  \item [($c_2$)] If $\pi\in{RD}_{n+2,k}$ and $\hat{\Omega}_{n+1}\circ\phi(\pi)\in\overline{{RFA}}_{n+1,k}$, then let
  $$\Omega_{n+2}(\pi)=
  \hat{\theta}^{-1}\circ\hat{\Omega}_{n+1}\circ\psi\circ\hat{\Omega}_{n}^{-1}\circ\hat{\psi}^{-1}\circ\hat{\Omega}_{n+1}\circ\phi(\pi);$$
  \item [($c_3$)] If $\pi\in\overline{{RD}}_{n+2,k}$, then let $$\Omega_{n+2}(\pi)=\hat{\theta}^{-1}\circ\hat{\Omega}_{n+1}\circ\theta(\pi).$$
\end{itemize}
\end{proof}

By Theorems \ref{thm01} and \ref{equidistributed}, we get
$$\sum_{\sigma\in\Q_{n+1}}x^{\asc(\sigma)}=\sum_{\sigma\in\Q_{n+1}}x^{n-\des(\sigma)}=\sum_{\sigma\in\Q_{n+1}}x^{\des(\sigma)}.$$
Hence $$\widetilde{A}_n(x)=x^n\widetilde{A}_n\left(\frac{1}{x}\right),$$
which implies that $\widetilde{A}_n(x)$ is symmetric.
\subsection{The $n$th-order recurrence relations}
\hspace*{\parindent}

Recall the following recurrence relation which is attributed to Euler (see~\cite{Hyatt16} for instance):
\begin{equation}\label{Anx-recu02}
A_n(x)=\sum_{k=0}^{n-1}\binom{n}{k}(x-1)^{n-k-1}A_k(x) \quad\text{for $n\geq1$}.
\end{equation}

As an analog of~\eqref{Anx-recu02}, we now present the following result.
\begin{thm}\label{thm-BinomialAnx-recu01}
The polynomials $\widetilde{A}_{n}(x)$ satisfy the recurrence relation
\begin{equation}\label{BinomialAnx-recu01}
\widetilde{A}_n(x)=\sum\limits_{j=1}^{n}{n\choose j}(x-1)^{j-1}\widetilde{A}_{n-j}(x)+x^n
\end{equation}
for $n\geq 1$, with the initial value $\widetilde{A}_{0}(x)=1$. Equivalently,
we have
\begin{equation}\label{BinomialAnx-recu02}
\widetilde{A}_n(x)=\sum\limits_{k=0}^{n-1}{n\choose k}(x-1)^{n-k-1}\widetilde{A}_{k}(x)+x^n.
\end{equation}
\end{thm}
\begin{proof}
Let $x$ be a positive integer. For any $n\geq 0$, let $\Q_{n+1}(x)$ be the set of pairs $(\pi,\phi)$ such that $\pi\in\Q_{n+1}$
 and $\phi$ is a map from $\Des(\pi)$ to $\{0,1,\ldots,x-1\}$. Thus
$$\widetilde{A}_n(x)=\sum\limits_{\pi\in\Q_{n+1}}x^{\des(\pi)}=|\Q_{n+1}(x)|.$$

For any $(\pi,\phi)\in\Q_{n+1}(x)$, there is a unique index $k\geq 1$ which satisfies
$\pi(k-1)<\pi(k)$ and $\pi(k)>\pi(k+1)>\cdots
>\pi(n+1)$. For the sequence $\pi(k),\pi(k+1),\ldots
,\pi(n+1)$, if
 $\phi(\pi(i))=0$ for some $k\leq i\leq n+1$, then let $k'$ be the
largest index in $\{k,k+1,\ldots,n+1\}$ such that $\phi(\pi(k'))=0$;
otherwise, let $k'=k$. Let $$\sigma=\pi(1),\pi(2),.\ldots,\pi(k')$$ and $$B=\{\pi(k'+1),\ldots,\pi(n+1)\}.$$
Then $\sigma$ is a permutation defined on the set $\{1,2,\ldots,n+1\}\setminus B$ and the entry $n+1$ appears as the first descent of $\sigma$ from left to right.

Now, we distinguish between the
following two cases:

{\bf Case 1.}$\pi(k')=n+1$ and $\phi(\pi(k'))\neq 0$.

Then the entry $n+1$ is the unique descent of the permutation $\pi$. Thus, we have $$\phi(\pi(i))\neq 0$$ for all $\pi(i)\in \Des(\pi)$. Note that $1\leq |B|\leq n$ and there are
${\binom{n}{|B|}}$ ways to form the set $B$.
Since $\Des(\pi)=\{n+1\}\cup (B\setminus\{\pi(n+1)\})$, there are $$(x-1)^{|\Des(\pi)|}=(x-1)^{|B|}$$ ways to form the map $\phi$.
This provides the term $\sum\limits_{B\subseteq [n]}(x-1)^{|B|}=x^n$.

{\bf Case 2.} Either $(i)$ $\pi(k')\neq n+1$ or $(ii)$ $\pi(k')=n+1$ and $\phi(\pi(k'))=0$.

Let
$$\red(\sigma):=\red(\sigma(1)),\red(\sigma(2)),\ldots ,\red(\sigma(k'))\in\mathfrak{S}_{k'},$$
where $\red$ is an increasing map from $\{\sigma(1),\sigma(2),\ldots,\sigma(k')\}$ to
$\{1,2,\ldots,k'\}$ such that $\red(\sigma(i))<\red(\sigma(j))$ if $\sigma(i)<\sigma(j)$ for all $i,j$.
Then the entry $k'$ is the first descent of the permutation $\red(\sigma)$ from left to right since $\red(n+1)=k'$ and $\red(\sigma)\in{\Q_{k'}}$.
Define a map $\phi':\Des(\red(\sigma))\mapsto\{0,1,\ldots,x-1\}$ by letting
 $$\phi'(i)=\phi(\red^{-1}(i))\text{ if }\red^{-1}(i)\neq\pi(k').$$
Then $(\red(\sigma),\phi')\in\Q_{k'}(x)$. Moreover, $\phi(i)\in\{1,2,\ldots,x-1\}$ for any $i\in B\setminus\{\pi(n+1)\}$.
Note that $1\leq |B|\leq n$, $k'=n+1-|B|$, there are
${\binom{n}{|B|}}$ ways to form the set $B$ and $\Q_{k'-1}(x)$ ways to form the pair $(\red(\sigma),\phi')$. Moreover, we have $\phi(i)\in\{1,2,\ldots,x-1\}$ for any $i\in B\setminus\{\pi(n+1)\}$. This provides the term $$\sum\limits_{j=1}^{n}{n\choose j}(x-1)^{j-1}\widetilde{A}_{n-j}(x).$$
Hence we derive the recurrence relation~\eqref{BinomialAnx-recu01}. Setting $k=n-j$ in~\eqref{BinomialAnx-recu01}, we immediately get~\eqref{BinomialAnx-recu02}. This completes the proof.

\end{proof}

Let $a_n=\sum_{\pi\in\Q_{n+1}}2^{\des(\pi)}$.
Note that
$$\widetilde{A}(2,z)=\frac{e^{2z}}{2-e^z}.$$
Let $\Stirling{n}{k}$ be the Stirling number of the second kind, which counts partitions of $[n]$ into $k$ nonempty subsets.
It is easy to verify that $a_n=2\sum_{k=0}^nk!\Stirling{n+1}{k+1}-1$. In particular, $a_0=1,a_1=3,a_2=11,a_3=51$.
The numbers $a_n$ have been studied by Gross~\cite{Gross}, Nelsen and Schmidt~\cite{Nelsen}.
It should be noted that
$a_n$ is the number of chains in power set of $[n]$ (see~\cite[A007047]{Sloane}). 

\begin{cor}
For $n\geq 1$, we have $$a_n=\sum\limits_{j=1}^{n}{n\choose j}a_{n-j}+2^n.$$
\end{cor}
\section{Multivariable binomial-Eulerian polynomials}\label{section03}
\hspace*{\parindent}

Let $\pi=\pi(1)\pi(2)\cdots\pi(n)\in\msn$.
An {\it excedance} in
$\pi$ is an index $i$ such that $\pi(i)>i$ and  a {\it fixed point} in
$\pi$ is an index $i$ such that $\pi(i)=i$. 
As usual, let $\exc(\pi)$, $\fix(\pi)$
and $\cyc(\pi)$ denote the number of excedances, fixed points and
cycles in $\pi$ respectively. For example, the permutation
$\pi=3142765$ has the cycle decomposition $(1342)(57)(6)$, so $\cyc(\pi)=3$, $\exc(\pi)=3$ and $\fix(\pi)=1$.
There is a large of literature devoted to various
generalizations and refinements of the joint distribution of
excedances and cycles, see, e.g.~\cite{Ksavrelof03,Ma16,Zhao13} and the references therein.

Define $$A_n(x,y,q)=\sum_{\pi\in\msn}x^{\exc(\pi)}y^{\fix(\pi)}q^{\cyc(\pi)}.$$
Let $A(x,y,q;z)=1+\sum_{n\geq 1}A_n(x,y,q)\frac{z^n}{n!}$.
Brenti~\cite[Proposition 7.3]{Bre00}
obtained that
\begin{equation*}\label{anxq-rr}
A(x,1,q;z)=\left(\frac{1-x}{e^{z(x-1)}-x}\right)^q.
\end{equation*}
Note that each object of $\msn$ is a disjoint union of one object counted by $A(x,0,q;z)$ and some fixed points.
Since each fixed point contributes no excedance but one cycle, by rules of exponential generating
function one has $A(x,1,q;z)=e^{qz}A(x,0,q;z)$ and $A(x,y,q;z)=e^{yqz}A(x,0,q;z)$.
Therefore,
\begin{equation}
A(x,y,q;z)=\left(\frac{1-x}{e^{z(x-y)}-xe^{(1-y)z}}\right)^q,
\end{equation}
which was also obtained by Ksavrelof and Zeng~\cite[p.~2]{Ksavrelof03}.
In the rest of this section, we study multivariable binomial-Eulerian polynomials.

A \emph{right-to-left maximum} of $\sigma\in\Q$ is an element $\sigma_i$ such that $\sigma_i>\sigma_j$ for every $j\in \{i+1,i+2,\ldots,n\}$ or $i=n$.
Let $\RLMAX(\sigma)$ denote the set of entries of right-to-left maxima of $\sigma$. Let $\rlmax(\sigma)=|\RLMAX(\sigma)|$.
For example, $\RLMAX(163254)=\{4,5,6\}$ and $\rlmax(163254)=3$.
A \emph{block} of $\sigma$ is a substring which ends with a right-to-left maximum, and contains exactly this one right-to-left maximum;
moreover, the substring is maximal, i.e., not contained in any larger such substring.
Clearly, any
permutation has a unique decomposition as a sequence of blocks. Let $\bk(\sigma)$ and $\bkone(\sigma)$ be the numbers of blocks and blocks of length one of $\sigma$, respectively.
Let $\fc(\sigma)$ be the length (number of terms) of the first block of $\sigma$ from left to right.
For example,
the block decomposition of $163254$ is given by $[16][325][4]$, $\bk(163254)=3,~\bkone(163254)=1$ and $\fbk(163254)=2$.

For any $\sigma\in\mathfrak{S}_{n}$, we can write $\sigma$ in standard cycle form satisfying the following conditions:
\begin{itemize}
  \item [\rm ($i$)] each cycle is end with its largest element;
  \item [\rm ($ii$)] the cycles are written in decreasing order of their largest element.
\end{itemize}

In the following discussion, we shall always write the cycle structure of $\sigma\in\msn$ in standard cycle form.
\begin{defn}
Let $\widehat{\mathcal{Q}}_{n}$ be the set of permutations of $[n]$ with the restriction that the sequence in the cycle containing $n$ is increasing.
\end{defn}
For example, $$\widehat{\mathcal{Q}}_3=\{(3)(2)(1),(2,3)(1),(3)(1,2),(1,3)(2),(1,2,3)\}.$$

Define $\widehat{\sigma}$ to be the word obtained from $\sigma\in\widehat{\Q}_n$ by writing it in standard cycle form and erasing
the parentheses. Then $\widehat{\sigma}\in\mathcal{Q}_n$. Thus, we get a bijection from $\widehat{\Q}_n$ to $\mathcal{Q}_n$.
Suppose that
$$\sigma=(\sigma_1,\sigma_2,\ldots,\sigma_{i_1})(\sigma_{i_1+1},\sigma_{i_1+2},\ldots,\sigma_{i_2})\cdots (\sigma_{i_{k-1}+1},\sigma_{i_{k-2}+2},\ldots,\sigma_{i_k})\in\widehat{\Q}_n.$$
Then
$\sigma_{i_1},\sigma_{i_2},\ldots,\sigma_{i_k}$ are the largest
elements of their cycles, and $\sigma_{i_1}>\sigma_{i_2}>\ldots>\sigma_{i_k}$.
Hence $\sigma(\sigma_i)>\sigma_i$ if and only if $\sigma_i<\sigma_{i+1}$. Let $\fc(\sigma)$ be the number of elements in the first cycle of $\sigma$.

From the above discussion, we can now conclude the following result.
\begin{prop}
For any $n\geq 1$, we have $$\sum_{\sigma\in\Q_{n}}x^{\asc(\sigma)}y^{\bkone(\sigma)}q^{\bk(\sigma)}p^{\fbk(\sigma)}=\sum_{\sigma\in\widehat{\Q}_{n}}x^{\exc(\sigma)}y^{\fix(\sigma)}q^{\cyc(\sigma)}p^{\fc(\sigma)}.$$
\end{prop}

Let $\widetilde{A}_n(x,y,q,p)=\sum_{\sigma\in\widehat{\Q}_{n+1}}x^{\exc(\sigma)}y^{\fix(\sigma)}q^{\cyc(\sigma)}p^{\fc(\sigma)}$.
The first few $\widetilde{A}_n(x,y,q,p)$ are given as follows:
\begin{align*}
\widetilde{A}_0(x,y,q,p)&=ypq,\\
\widetilde{A}_1(x,y,q,p)&=y^2pq^2+xp^2q,\\
\widetilde{A}_2(x,y,q,p)&=pq^3y^3+pq^2xy+2p^2q^2xy+p^3qx^2.
\end{align*}

\begin{thm}
Let
$\widetilde{A}(x,y,q,p;z)=\sum_{n\geq 0}\widetilde{A}_n(x,y,q,p)\frac{z^n}{n!}$.
We have
\begin{equation}\label{EGF-A}
\widetilde{A}(x,y,q,p;z)=\left(e^{xpz}+y-1\right)pqA(x,y,q;z).
\end{equation}

\end{thm}
\begin{proof}
Let $n$ be a fixed positive integer. Given $\pi\in \widehat{\Q}_{n+1}$.
Suppose the first cycle of $\pi$ is given by $\sigma=(c_1,c_2,\ldots,c_{k},n+1)$. So $\pi$ can be split into the
cycle $\sigma$ and a permutation $\tau$ on the set
 $\{1,2,\ldots,n+1\}\setminus\{c_1,c_2,\ldots,c_k,n+1\}$, i.e., $\pi=\sigma \cdot\tau$.
When $k=0$, we have $$\exc(\pi)=\exc(\tau), \fix(\pi)=\fix(\tau)+1, \cyc(\pi)=\cyc(\tau)+1,\fc(\pi)=1.$$
This provides the term $ypqA_{n}(x,y,q)$.
When $1\leq k\leq n$,
there are $\binom{n}{k}$ ways to form the set $\{c_1,c_2,\ldots,c_k\}$. Moreover, we have
$$\exc(\pi)=\exc(\tau)+k, \fix(\pi)=\fix(\tau), \cyc(\pi)=\cyc(\tau)+1,\fc(\pi)=k+1.$$
This provides the term $\sum\limits_{k=1}^n{n\choose{k}}x^{k}qp^{k+1}A_{n-k}(x,y,q)$.
Therefore, we obtain
\begin{equation}\label{Rec-A}
\widetilde{A}_{n}(x,y,q,p)=ypqA_{n}(x,y,q)+\sum\limits_{k=1}^n{n\choose{k}}x^{k}qp^{k+1}A_{n-k}(x,y,q).
\end{equation}
Multiplying both sides of~\eqref{Rec-A} by $z^n/n!$ and summing over all nonnegative integers $n$, we get that
\begin{align*}
\widetilde{A}(x,y,q,p;z)&=ypqA(x,y,q;z)+pq\sum_{n=1}^{\infty}\sum_{k=1}^{n}\binom{n}{k}(xp)^kA_{n-k}(x,y,q)\frac{z^n}{n!}\\
&=ypqA(x,y,q;z)+pq\sum_{n=1}^{\infty}\sum_{k=0}^n\binom{n}{k}(xp)^kA_{n-k}(x,y,q)\frac{z^n}{n!}-pq\left((A(x,y,q;z)-1\right)\\
&=ypqA(x,y,q;z)+pq\left(e^{xpz}A(x,y,q;z)-1\right)-pq\left((A(x,y,q;z)-1\right)\\
&=\left(e^{xpz}+y-1\right)pqA(x,y,q;z).
\end{align*}
This completes the proof.
\end{proof}

From~\eqref{EGF-A}, we see that
\begin{align*}
\widetilde{A}(x,1,-1,-1;z)&=e^{-xz}A(x,1,-1;z)=\frac{e^{-z}-xe^{-xz}}{1-x},\\
\widetilde{A}(x,1,-1,1;z)&=-e^{xz}A(x,1,-1;z)=\frac{e^{2xz-z}-xe^{xz}}{x-1}.\\
\end{align*}
It is routine to check that
\begin{align*}
\frac{e^{-z}-xe^{-xz}}{1-x}&=\sum_{n=0}^{\infty}(-1)^n\frac{1-x^{n+1}}{1-x}\frac{z^n}{n!},
\end{align*}
\begin{equation}\label{betti-GF01}
\frac{e^{(2x-1)z}-xe^{xz}}{x-1}=\sum_{n=0}^{\infty}\frac{(1-2x)^{2n}-x^{2n+1}}{x-1}\frac{z^{2n}}{(2n)!}
+\sum_{n=1}^{\infty}\frac{(1-2x)^{2n-1}+x^{2n}}{1-x}\frac{z^{2n-1}}{(2n-1)!}.
\end{equation}

Therefore, we get the following corollary.
\begin{cor}\label{cor3.3}
For $n\geq 0$, we have
\begin{align*}
\widetilde{A}_n(x,1,-1,-1)&=\sum_{\sigma\in\widehat{\Q}_{n+1}}x^{\exc(\sigma)}(-1)^{\cyc(\sigma)+\fc(\sigma)}=(-1)^n(1+x+x^2+\cdots+x^n);\\
\widetilde{A}_n(x,1,-1,1)&=\sum_{\sigma\in\widehat{\Q}_{n+1}}x^{\exc(\sigma)}(-1)^{\cyc(\sigma)}
=\sum_{k=0}^nx^{n-k}\sum_{i=k}^n(-1)^{i-1}2^{n-i}\binom{n}{i}.
\end{align*}
\end{cor}
It would be interesting to present a combinatorial proof of Corollary~\ref{cor3.3}.

Let $$B(n,k)=\sum_{i=k}^n(-1)^{k-i}2^{n-i}\binom{n}{i}.$$
It should be noted that the numbers $B(n,k)$ are known as the $(k-2)$-nd {\it Betti numbers} of the complement
of the $k$-equal real hyperplane arrangement in $\mathbb{R}^n$ (see~\cite[Theorem~4.1.5]{Green09} for instance).
The Betti number $B(n,i)$ was first studied by Bj\"{o}rner and Welker~\cite{Bjorner}, and subsequently studied by Green~\cite{Green09,Green10}. The reader is referred to Green~\cite[page~1038]{Green10} for various interpretations of the numbers $B(n,i)$.

From Corollary~\ref{cor3.3}, we see that
\begin{equation}\label{betti}
\sum_{\sigma\in\widehat{\Q}_{n+1}}x^{\exc(\sigma)}(-1)^{\cyc(\sigma)}
=\sum_{k=0}^n(-1)^{k+1}B(n,k)x^{n-k}.
\end{equation}

An {\it anti-excedance} in
$\pi\in\msn$ is an index $i$ such that $\pi(i)\leq i$.
Let $\aexc(\pi)$ be the number of anti-excedances of $\pi$.
Clearly, $\exc(\pi)+\aexc(\pi)=n$ for $\pi\in\msn$. For $\pi\in\widehat{\Q}_{n+1}$, if $\exc(\pi)=n-k$, then $\aexc(\pi)=k+1$.
Therefore, using~\eqref{betti}, we get the following result.
\begin{thm}\label{Betticombinatorial}
For $n\geq 0$, we have
$$B(n,k)=\sum_{\substack{\pi\in\widehat{\Q}_{n+1} \\ \exc(\pi)=n-k}}(-1)^{\cyc(\pi)+\aexc(\pi)}.$$
\end{thm}

Using Theorem~\eqref{Betticombinatorial}, one may introduce some $q$-analogs of the Betti numbers $B(n,k)$.

Let $B_n(x)=\sum_{k=0}^nB(n,k)x^{k}$.
Combining~\eqref{betti-GF01} and~\eqref{betti}, we obtain the following result.
\begin{prop}\label{BettiGF}
We have
$$\sum_{n\geq 0}B_n(x)\frac{z^n}{n!}=\frac{e^z+xe^{(2+x)z}}{1+x}.$$
\end{prop}
%

Define $$T_n(q)=\sum_{\sigma\in\widehat{\Q}_{n+1}}q^{\cyc(\sigma)}=\sum_{k=1}^{n+1}T(n,k)q^k.$$
Let $T(q,z)=\sum_{n\geq 0}T_n(q)\frac{z^n}{n!}$.
It follows from~\eqref{EGF-A} that
\begin{equation}\label{Tqz}
T(q,z)=qe^{z}\sum_{n\geq 0}\sum_{k=0}^n\stirling{n}{k}q^k\frac{z^n}{n!}=\frac{qe^{z}}{(1-z)^q},
\end{equation}
where $\stirling{n}{k}$ is the signless Stirling number of the first kind, i.e., the number of permutations of $\msn$ with $k$ cycles.
Using~\eqref{Tqz}, we immediately get the following result.
\begin{prop}
For $n\geq 2$, we have
$T_n(-1)=\sum_{\sigma\in\widehat{\Q}_{n+1}}(-1)^{\cyc(\sigma)}=n-1$.
\end{prop}

Let $F_n(q)=\sum_{k=0}^n\stirling{n}{k}q^k$.
Combining~\eqref{Tqz} and the well known recurrence relation $F_n(q)=(n-1+q)F_{n-1}(q)$, one can easily derive that
the polynomials $T_n(q)$ satisfy the recurrence relation
\begin{equation}\label{Tnq-recu}
T_{n+1}(q)=(n+1+q)T_n(q)-nT_{n-1}(q),
\end{equation}
with the initial conditions
$T_0(q)=q,~T_1(q)=q+q^2$. Equivalently,
we have
\begin{equation*}\label{Tnk-recu}
T(n+1,k)=(n+1)T(n,k)+T(n,k-1)-nT(n-1,k).
\end{equation*}

Recall that
the Charlier polynomials are defined by
$$C_n^{(a)}(x)=\sum_{k=0}^n(-a)^{n-k}\binom{n}{k}\binom{x}{k}k!,~a\neq 0.$$
These polynomials are generated by
${e^{-az}}{(1+z)^x}=\sum_{n\geq 0}C_n^{(a)}(x)\frac{z^n}{n!}$.
Hence $$T_n(q)=(-1)^nqC_n^{(1)}(-q)=q\sum_{k=0}^n(-1)^k\binom{n}{k}\binom{-q}{k}k!.$$
It is well known that Charlier polynomials are orthogonal polynomials and have only real zeros. Hence the polynomial
$T_n(q)$ has only real zeros for any $n\geq 0$.

\end{document}